\documentclass[a4paper,11pt,reqno]{amsart}
\pdfoutput=1

\usepackage[latin1]{inputenc}
\usepackage{amssymb}
\usepackage{enumerate}
\usepackage{amsmath}
\usepackage{amsthm}
\usepackage{graphicx}
\usepackage{hyperref}
\usepackage{stmaryrd}
\usepackage{xcolor}
\usepackage{url}
\usepackage{mathdots} %for antidiagonal dots

\graphicspath{{figures/}}

\newtheorem{thm}{Theorem}[section]
\newtheorem{defi}[thm]{Definition}
\newtheorem{prop}[thm]{Proposition}
\newtheorem{lem}[thm]{Lemma}
\newtheorem{cor}[thm]{Corollary}
\newtheorem{conj}[thm]{Conjecture}
\newtheorem{rem}[thm]{Remark}
\newtheorem{ex}[thm]{Example}

\newcommand{\ar}{\textnormal{area}}
\newcommand{\Id}{\textnormal{Id}}
\renewcommand{\S}{\mathbf{S}}
\newcommand{\M}{\mathbf{M}}
\newcommand{\D}{\mathbf{D}}
\newcommand{\s}{\mathfrak{s}}
\newcommand{\Z}{\mathbb{Z}}

\newcommand{\st}{(\mathbf{s},\mathbf{t})}
\newcommand{\overDelta}[1]{\overline{\Delta}_{\textit{ }#1}}
\newcommand{\underDelta}[1]{\underline{\Delta}_{\textit{ }#1}}
\newcommand{\OP}{\textnormal{Op}}
\newcommand{\CT}{\textnormal{CT}}
\newcommand{\Sgrp}{\mathfrak{S}}
\newcommand{\sgn}{\textnormal{sgn}}

\newcommand{\red}[1]{{\color{red}{#1}}}

\newcommand{\green}[1]{{\color{green}{#1}}}

\title[Refined enumerations of ASTs]{Refined enumerations of alternating sign triangles}
\author{Florian Aigner}
\address{Florian Aigner, Universit\"at Wien, Fakult\"at f\"ur Mathematik, Oskar-\linebreak Morgenstern-Platz~1, 1090 Wien, Austria}
\email{florian.aigner@univie.ac.at}
\thanks{Supported by the Austrian Science Fund FWF, START grant Y463.}

\begin{document}

%%%%%%%%%%%%%%%%%%%%%%%%%%%%%%%%%%%%%%%%%%%%%%%%%%%%%%%%%%%%%%%%%%%%%%%%%%%%%%%%%%%%%%%%%%%%%%%%%%%%%%
%%%%%%%%%%%%%%%%%%%%%%%%%%%%%%%%%%%%%%%%%%%%%% Abstract %%%%%%%%%%%%%%%%%%%%%%%%%%%%%%%%%%%%%%%%%%%%%%
%%%%%%%%%%%%%%%%%%%%%%%%%%%%%%%%%%%%%%%%%%%%%%%%%%%%%%%%%%%%%%%%%%%%%%%%%%%%%%%%%%%%%%%%%%%%%%%%%%%%%%

\begin{abstract}
This article introduces and investigates a refinement of alternating sign trapezoids by means of Catalan objects and Motzkin paths. Alternating sign trapezoids are a generalisation of alternating sign triangles, which were recently introduced by Ayyer, Behrend and Fischer. We show that the number of alternating sign trapezoids associated with a Catalan object (resp. a Motzkin path) is a polynomial function in the length of the shorter base of the trapezoid. We also study the rational roots of these polynomials and formulate several conjectures and derive some partial results. Lastly, we deduce a constant term identity for the refined counting of alternating sign trapezoids.
\end{abstract}

%%%%%%%%%%%%%%%%%%%%%%%%%%%%%%%%%%%%%%%%%%%%%%%%%%%%%%%%%%%%%%%%%%%%%%%%%%%%%%%%%%%%%%%%%%%%%%%%%%%%%%
%%%%%%%%%%%%%%%%%%%%%%%%%%%%%%%%%%%%%%%%%%%% Introduction %%%%%%%%%%%%%%%%%%%%%%%%%%%%%%%%%%%%%%%%%%%%
%%%%%%%%%%%%%%%%%%%%%%%%%%%%%%%%%%%%%%%%%%%%%%%%%%%%%%%%%%%%%%%%%%%%%%%%%%%%%%%%%%%%%%%%%%%%%%%%%%%%%%

\maketitle

\section{Introduction}
In the 1980s Mills, Robbins and Rumsey \cite{MillsRobbinsRumsey82,RobbinsRumsey86} introduced \emph{alternating sign matrices} (ASMs) and thereby founded a new branch in combinatorics. Over time, more combinatorial objects, equinumerous to ASMs, were introduced, e.g.\ \emph{fully packed loops} (FPLs). FPLs led naturally to a refined enumeration using Catalan objects. This refinement gave rise to a variety of important results, one of them is the Razumov-Stroganov-Cantini-Sportiello Theorem \cite{CantiniSportiello11,RazumovStroganov01}. In \cite{AyyerBehrendFischer1611.03823}, Ayyer, Behrend and Fischer introduced  \emph{alternating sign triangles} (ASTs) and showed that they are equinumerous to ASMs. While there exists an easy bijection between ASMs and FPLs, finding a bijection between ASMs and ASTs is still an open problem.\\

In this paper we introduce and investigate a refined enumeration of ASTs using \emph{centred Catalan sets}, objects enumerated by the Catalan numbers and Motzkin paths (Ayyer \cite{Ayyer}). As we will see in Section \ref{sec: roots}, it makes sense to look at both refinements, even though the Motzkin path refinement is coarser, since we may deduce statements and make conjectures for this refinement that are not true for the centred Catalan set refinement.

The structure of ASTs and their associated centred Catalan set lead naturally to the introduction of \emph{AS-trapezoids}. AS-trapezoids arise by generalising ASTs from a triangular to a trapezoidal shape, and have by definition the following property: by taking an AST of order $n$ and putting an $(m,2n)$-AS-trapezoid on top of it, one obtains an AST of order $m+n$. In addition, the associated centred Catalan set (resp. Motzkin path) of the AST of order $m+n$ is the concatenation of the centred Catalan set (resp. Motzkin path) associated with the AST of order $n$ and the $(m,2n)$-AS-trapezoid. In Lemma \ref{lem: splitting lemma} we show that the converse is also true. In particular, every AST whose associated centred Catalan set (resp. Motzkin path) can be split into two parts can also be split into two parts: a smaller AST corresponding to the first centred Catalan set (resp. Motzkin path) and an AS-trapezoid corresponding to the second. 
 Following \cite{Ayyer}, we define the weight function $w_l(S)$ (resp. $w_l(M)$) as the number of $(n,l)$-AS-trapezoids with associated centred Catalan set $S$ of size $n-1$ (resp. Motzkin path $M$ of length $n$). Exploiting the splitting property of ASTs into an AST and an AS-trapezoid leads to our first main result.

\begin{prop}
\label{prop: multiplicity of weights}
Let $S_1,S_2$ be centred Catalan sets and $M_1, M_2$ be Motzkin paths. We have
\begin{align}
\label{eq: multiplicity of Catalan weights}
w_l(S_1 \circ S_2)&=w_l(S_1)w_{l+2|S_1|-2}(S_2),\\
\label{eq: multiplicity of Motzkin weights}
w_l(M_1 \circ M_2)&=w_l(M_1)w_{l+2|M_1|}(M_2),
\end{align}
i.e., the weight functions are multiplicative.
\end{prop}

In Theorem \ref{thm: AS-ts} we present an enumeration formula for the number of AS-trapezoids. This formula was first conjectured by Behrend \cite{ASt-conj} and later independently by the author. The proof was found by Behrend and Fischer \cite{ASt-conj} using the six-vertex model and by Fischer \cite{Fischer1804.08681} using the operator formula.\\

We establish a bijection between the AS-trapezoids corresponding to a given centred Catalan set $S$ to $\st$-trees, which were defined in \cite{Fischer1804.03630}, but had already appeared as \emph{partial monotone triangles} in \cite{Fischer11}, and are a generalisation of monotone triangles. By analysing the structure of the $\st$-trees, we can deduce our second main result.

\begin{thm}
 \label{thm: Catalan weight is polynomial}
 Let $S=S_1 \circ \ldots \circ S_k$ be a centred Catalan set and $S_1,\ldots, S_k$ its irreducible components. The weight $w_l(S)$ is a polynomial in $l$ of degree $|\lambda(S)/\mu(S)|=\ar(\M(S))$ with leading coefficient $\prod_{i=1}^k\frac{f^{\lambda(S_i)/ \mu(S_i)}}{|\lambda(S_i)/\mu(S_i)|!}$, where $f^{\lambda/ \mu}$ denotes the number of standard Young tableaux of skew shape $\lambda/ \mu$.
\end{thm}

Remarkably, there exists an analogous theorem for the refined enumeration of FPLs for a specific class of non-crossing matchings. This theorem was conjectured in \cite{Zuber04} and proved in \cite{Aigner_FPL, CaselliKrattenthalerLassNadeau04}. A detailed explanation is provided in Remark \ref{rem: ASTs and FPLs}.\\

Using the enumeration formula for $\st$-trees from \cite{Fischer1804.03630}, we can deduce a constant term formula for the refined enumeration of AS-trapezoids. Essentially, this formula generalises a constant term formula for refined ASTs \cite[Theorem 2]{Fischer1804.03630}.

\begin{thm}
\label{thm: constant term identity}
Let $S=\{s_1, \ldots, s_u,0,s_{u+1}, \ldots, s_n\}$ be a centred Catalan set with $s_1< \ldots < s_u \leq -1, 1 \leq s_{u+1} < \ldots < s_n$. The number $w_l(S)$ of $(n,l)$-AS-trapezoids with centred Catalan set $S$ is equal to the constant term of
\[
\prod_{i=1}^n x_i^{-n-s_i}\prod_{i=u+1}^n x_i^{2-l}(1+x_i)^l
\prod_{1\leq i < j \leq n}(x_j-x_i)(1+x_i+x_ix_j),
\]
in $x_1,\ldots, x_n$.
\end{thm}

The final section of this paper describes the rational roots of the weight functions $w_l(S)$ and $w_l(M)$. Both seem to have a rich structure of rational roots that is currently not completely understood. We present our first analysis of their structure in the form of conjectures and partial results. The goal is to find a description of all rational roots using combinatorial methods, similar to \cite{FonsecaNadeau11} for the number of FPLs associated with certain families of link patterns.\\

The paper is structured as follows. In Section \ref{sec: preliminaries} we define centred Catalan sets, Motzkin paths, ASTs and AS-trapezoids and establish the relationships between them. In Section \ref{sec: structure of AS-ts}, we present a relation between the inner structure of an AS-trapezoid and the centred Catalan path it corresponds to. This allows us to deduce the splitting lemma (Lemma \ref{lem: splitting lemma}), which leads directly to our first main result.
In Section \ref{sec: refined enum of AS-ts}, we relate AS-trapezoids to $\st$-trees and derive our second main result. Section \ref{sec: const term} contains the derivation of the constant term identity for the refined enumeration of AS-trapezoids. Finally, in Section \ref{sec: roots} we present conjectures and partial results concerning the rational roots of $w_l(S)$ and $w_l(M)$.\\

An extended abstract of this paper was published in the proceedings of FPSAC 2017 \cite{Aigner_AST0}.

%%%%%%%%%%%%%%%%%%%%%%%%%%%%%%%%%%%%%%%%%%%%%%%%%%%%%%%%%%%%%%%%%%%%%%%%%%%%%%%%%%%%%%%%%%%%%%%%%%%%%%
%%%%%%%%%%%%%%%%%%%%%%%%%%%%%%%%%%%%%%%%%%%% Preleminaries %%%%%%%%%%%%%%%%%%%%%%%%%%%%%%%%%%%%%%%%%%%
%%%%%%%%%%%%%%%%%%%%%%%%%%%%%%%%%%%%%%%%%%%%%%%%%%%%%%%%%%%%%%%%%%%%%%%%%%%%%%%%%%%%%%%%%%%%%%%%%%%%%%

\section{Preliminaries}
\label{sec: preliminaries}

\subsection{Centred Catalan sets, Dyck and Motzkin paths}

We introduce new combinatorial objects that are enumerated by the Catalan numbers and explore their relationship to Dyck paths.

%~~~~~~~~~~~~~~~~~~~~~~~~~~~~~~~~~~~~~~~~~~~~~~~~~~~~~~~~~~~~~~~~~~~~~~~~~~~~~~
%~~~~~~~~~~~~~~~~~~~~~~~~~~~~ centred Catalan sets ~~~~~~~~~~~~~~~~~~~~~~~~~~~~
%~~~~~~~~~~~~~~~~~~~~~~~~~~~~~~~~~~~~~~~~~~~~~~~~~~~~~~~~~~~~~~~~~~~~~~~~~~~~~~

\begin{defi}
A \emph{centred Catalan set} $S$ of size $n$ is an $n$-subset of $\{-n+1,-n+2,\ldots, n-1\}$ such that $|S \cap \{-i,-i+1, \ldots, i\}| \geq i+1$ for all $0 \leq i \leq n-1$; in particular $0 \in S$.
\end{defi}

\begin{ex}
The centred Catalan sets of size $3$ are
\begin{align*}
\{0,1,2\}, \quad
\{0,-1,2\}, \quad
\{0,1,-2\}, \quad
\{0,-1,-2\}, \quad
\{-1,0,1\}.
\end{align*}
\end{ex}

Centred Catalan sets of size $n$ are in bijection with Dyck paths of length $2n$. For a given centred Catalan set $S$, we can construct a Dyck path $\D(S)$ in the following way. We read the integers $-n+1, \ldots, n-1,n$ in the order $0,-1,1,-2,2,\ldots,-n+1,n-1,n$ and draw a north-east step if the number is in $S$ and a south-east step otherwise. For an example see Figure \ref{fig: cCs and Dyck}.
In fact, there are $2^{n-1}$ different bijections of the above kind between centred Catalan sets of size $n$ and Dyck paths of length $2n$. For every $1 \leq i \leq n-1$, we can switch the order of reading $-i,i$ in the above algorithm and obtain a new bijection.\\

\begin{figure}
$\{-3,-1,0,1,3,4\} \quad \Leftrightarrow \quad$
 \centering
 \includegraphics[scale=1.5]{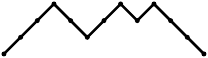}
 \caption{\label{fig: cCs and Dyck} A centred Catalan set of size $6$ and its corresponding Dyck path.}
\end{figure}

Let $l$ be an integer and define the dilation operator $\s_l: \Z \rightarrow \Z$ as
\[
\s_l(x)=
\begin{cases}
x+l \quad &x>0,\\
0 & x=0,\\
x-l & x<0.
\end{cases}
\]
By abuse of notation we write $\s_l: 2^\Z \rightarrow 2^\Z, \;\s_l(A)=\{\s_l(x)|x\in A\}$. The \emph{concatenation} $S_1 \circ S_2$ of two centred Catalan sets $S_1,S_2$ is defined as
\[
S_1 \circ S_2 := S_1 \cup \s_{|S_1|-1}(S_2).
\]

\begin{ex}
The concatenation of the centred Catalan sets $S_1 = \{-1,0,1\}$ and $S_2=\{-1,0,1,2\}$ is
\[
S_1 \circ S_2 = \{-1,0,1\} \cup \s_{2}(\{-1,0,1,2\}) = \{-3,-1,0,1,3,4\}.
\]
\end{ex}

We call a centred Catalan set $S$ \emph{irreducible} if there exist no centred Catalan sets $S_1,S_2$ of size at least $2$ such that $S=S_1 \circ S_2$. It is not hard to convince oneself that every centred Catalan set can be written uniquely as the concatenation of irreducible centred Catalan sets.
For two centred Catalan sets $S_1,S_2$, the Dyck path $\D(S_1 \circ S_2)$ is obtained by deleting the last step of $\D(S_1)$ and the first step of $\D(S_2)$, and concatenating both paths, note that one does not simply add the paths $\D(S_1)$ and $\D(S_2)$, see Figure \ref{fig: Splittings}.\\

%~~~~~~~~~~~~~~~~~~~~~~~~~~~~~~~~~~~~~~~~~~~~~~~~~~~~~~~~~~~~~~~~~~~~~~~
%~~~~~~~~~~~~~~~~~~~~~~~~~~~~ Motzkin paths ~~~~~~~~~~~~~~~~~~~~~~~~~~~~
%~~~~~~~~~~~~~~~~~~~~~~~~~~~~~~~~~~~~~~~~~~~~~~~~~~~~~~~~~~~~~~~~~~~~~~~

A \emph{Motzkin path} of length $n$ is a path on the half-plane $y \geq 0$ starting at $(0,0)$ and ending at $(n,0)$ with step-set $\{(1,1),(1,0),(1,-1)\}$. In the following, we encode a Motzkin path $M$ of length $n$ with a sequence $M=(m_1,\ldots,m_n)$ with $m_i \in \{1,0,-1\}$ for $1 \leq i \leq n$, where $m_i$ corresponds to the step $(1,m_i)$. The concatenation of two Motzkin paths is given by attaching the second path at the end of the first one. A Motzkin path is called \emph{irreducible} iff it cannot be written as a concatenation of non-empty Motzkin paths.

\begin{ex}
The following are all Motzkin paths of length $3$, represented graphically and as sequences.
\[
\begin{array}{cccc}
\includegraphics[scale=1]{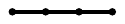} & \includegraphics[scale=1]{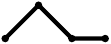} & \includegraphics[scale=1]{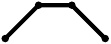} & \includegraphics[scale=1]{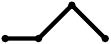}\\
(0,0,0) & (1,-1,0) & (1,0,-1) & (0,1,-1)
\end{array}
\]
\end{ex}

We define $\M(D)$ to be the Motzkin path of size $n-1$ obtained by ``averaging'' the steps in the Dyck path $D$ of size $2n$: the $i$-th step of $\M(D)$ is the average of the $(2i)$-th and $(2i+1)$-st step of $D$. By averaging we mean that two north-east steps result in a north-east step, a north-east step and a south-east step in an east step and two south-east steps in a south-east step. This map is a surjection from Dyck paths of length $2n$ to Motzkin paths of length $n-1$. 
For a centred Catalan set $S$ of size $n$, the Motzkin path $\M(S):=\M(\D(S))=(m_1(S), \ldots,m_{n-1}(S))$ is given by
\[
m_i(S):= |\{-i,i\} \cap S|-1.
\]
It is easy to see that concatenating centred Catalan sets (resp. Motzkin paths) commutes with the map from centred Catalan sets to Motzkin paths, i.e., $\M(S_1 \circ S_2)= \M(S_1)\circ \M(S_2)$. Further every irreducible component of $\M(S)$ corresponds to an irreducible component of $S$ and vice-versa. For an example see Figure \ref{fig: Splittings}.\\

\begin{figure}
 \centering
 \includegraphics[width=.9\textwidth]{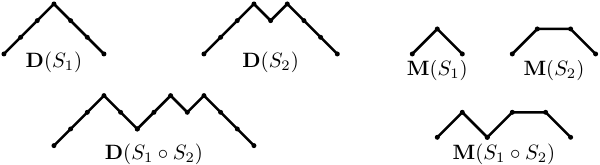}
  \caption{\label{fig: Splittings} The Dyck and Motzkin paths of the centred Catalan set $S=\{-3,-1,0,1,3,4\}$ and its irreducible components $S_1=\{-1,0,1\}$ and $S_2=\{-1,0,1,2\}$.
  }
\end{figure}

%~~~~~~~~~~~~~~~~~~~~~~~~~~~~~~~~~~~~~~~~~~~~~~~~~~~~~~~~~~~~~~
%~~~~~~~~~~~~~~~~~~~~~~~~~~~~ ASTs ~~~~~~~~~~~~~~~~~~~~~~~~~~~~
%~~~~~~~~~~~~~~~~~~~~~~~~~~~~~~~~~~~~~~~~~~~~~~~~~~~~~~~~~~~~~~

\subsection{Alternating sign triangles}
Alternating sign triangles were recently introduced by Ayyer, Behrend and Fischer \cite{AyyerBehrendFischer1611.03823}, and are the latest addition to a family of combinatorial objects with the same enumeration formula.

 \begin{defi}
An \emph{alternating sign triangle (AST)} of order $n$ is a configuration of $n$ centred rows such that the $i$-th row, counted from the bottom, has $2i-1$ elements. The entries are $-1$, $0$, or $1$ such that in all rows and columns the non-zero elements are alternating, all row sums are $1$ and in every column the first non-zero entry from the top is positive.
 \end{defi}

  The following is an example of an AST of order $6$
 \[
\begin{array}{c	c c c c c c c c c c}
0 & 0 & 0& 0 & 0 & 0 & 0 & 0 & 0 & 1 & 0\\
  & 0 & 0 & 1 &0 & 0 & 0 & 0 & 0 &0\\
  &   &   1 & -1 & 0& 0  & 0 & 0 &1\\
  &   &     &  0  & 0 & 1 & 0 & 0\\
  &   &     &     & 1 & -1 & 1\\
  & & & & &1
\end{array}.
\]

 \begin{thm}[{\cite{AyyerBehrendFischer1611.03823}}]
 The number of ASTs of order $n$ is given by
 \[
 \prod_{i=0}^{n-1} \frac{(3i+1)!}{(n+i)!},
 \]
 i.e., ASTs of order $n$ and $n \times n$ ASMs are equinumerous.
 \end{thm} 

 We label the columns of an AST $A$ of order $n$ from left to right with $-n+1, \ldots, n-1$ and the rows from bottom to top with $1, \ldots, n$.

\begin{prop}
\label{prop: AST to cCS}
 Let $A$ be an AST of order $n$ and $\S(A)$ be the set of columns with column sum equal to $1$. Then $\S(A)$ is a centred Catalan set of size $n$. Conversely, for all centred Catalan sets $S$ of size $n$ there exists an AST $A$ of order $n$ with $\S(A)=S$.
\end{prop}
\begin{proof}
The definition of an AST implies that the column sums can be either $0$ or $1$. Since all row sums are equal to $1$, the sum over all entries in an AST is equal to $n$. Therefore, exactly $n$ columns have column sum equal to $1$, i.e., the set $\S(A)$ is an $n$-subset of $\{-n+1,\ldots,n-1\}$. Define $\S_i(A)$ to be the set of columns $j$ such that $|j|<i$ and the partial column sum of elements below the $(i+1)$-st row is equal to $1$. If a column $j$ is an element of $\S_i(A)$ for an $i>|j|$, it follows that it has positive column sum. Hence, we have the following relation between $\S_i(A)$ and $\S(A)$
\begin{align*}
\S_{i}(A) &\subseteq \S(A) \cap \{-i+1,-i+2, \ldots,i-1\}.
\end{align*}
Since the partial column sums can only have a value of $1$, $0$, or $-1$, and the sum of the partial column sums of the elements below the $(i+1)$-st row is $i$, we obtain $|\S_i(A)| \geq i$.
Hence we have
\[
i \leq |\S_i(A)| \leq |\S(A) \cap \{-i+1, \ldots, i-1\}|,
\]
which implies the first claim.

Now let $S$ be given. By definition we can choose a sequence $(s_i)_{1 \leq i \leq n}$ such that $S=\{s_i: 1 \leq i\leq n\}$ and $|s_i| < i$. We construct an AST $A$ by setting all entries equal to $0$ except the entry in column $s_i$ of the $i$-th row for $1 \leq i \leq n$, which we set to $1$. Then $A$ is an AST and satisfies $\S(A)=S$.
\end{proof}

The following refinement of ASTs by Motzkin paths is due to Ayyer \cite{Ayyer}.

\begin{cor}
 \label{cor: Motzkin paths and ASTs}
  The map $\M(A):=\M(\S(A))$ is a surjection from ASTs of order $n$ to Motzkin paths of length $n-1$.
\end{cor}

\begin{figure}
 \centering
 \includegraphics[scale=1.1]{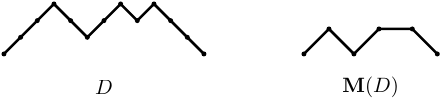}
 \caption{\label{fig: Dyck and Motzkin} A Dyck path and its corresponding Motzkin path.}
\end{figure}

We are interested in the \emph{weight function} $w(S)$ (resp. $w(M)$) of a centred Catalan set $S$ (resp. Motzkin path $M$), which is defined as the number of ASTs $A$ with $\S(A)=S$ (resp. $\M(A)=M$).

%~~~~~~~~~~~~~~~~~~~~~~~~~~~~~~~~~~~~~~~~~~~~~~~~~~~~~~~~~~~~~~
%~~~~~~~~~~~~~~~~~~~~~~~~~~~~ ASts ~~~~~~~~~~~~~~~~~~~~~~~~~~~~
%~~~~~~~~~~~~~~~~~~~~~~~~~~~~~~~~~~~~~~~~~~~~~~~~~~~~~~~~~~~~~~

\subsection{Alternating sign trapezoids}
Following the definition of ASTs, we introduce a generalisation of ASTs to a trapezoidal shape.

\begin{defi}
Let $n,l$ be positive integers. An $(n,l)$-\emph{AS-trapezoid} is an array of $n$ centred rows such that the $i$-th row from the bottom has $l+2i-1$ entries, which are filled with $-1$, $0$, or $1$, such that the following hold: all row sums are equal to $1$, the column sums are equal to $0$ for the central $l-1$ columns, the non-zero entries in all rows and columns are alternating and in every column the first non-zero entry from the top is positive.
\end{defi}

\emph{Alternating sign trapezoids} were first introduced in \cite{Aigner_AST0} with bases of odd length. The above definition is more general, in particular it allows bases of even length. The term $(n,l)$-AS-trapezoid in \cite{Aigner_AST0} corresponds to the term $(n,2l)$-AS-trapezoid in this article.\\

ASTs of order $n+1$ and $(n,2)$-AS-trapezoids are in bijection by deleting the bottom $1$ of an AST or adding a $1$ at the bottom of an $(n,2)$-AS-trapezoid. We label the rows of an $(n,l)$-AS-trapezoid from bottom to top with $1,\ldots, n$ and the columns from left to right by $-n+1, \ldots, l+n-1$. Let $A$ be an $(n,l)$-AS-trapezoid. We define $\S(A)$ to be the set obtained as follows: First, take the set of columns of $A$ with positive column sum, then subtract $1$ from the columns with non-positive label. Now, subtract $l-1$ from the columns with positive labels, and finally add $0$ to the set. Analogous to Proposition \ref{prop: AST to cCS}, one can prove that $\S(A)$ is a centred Catalan set of size $n+1$.\\

The following is an example of a $(4,6)$-AS-trapezoid $A$ with $\S(A)=\{-2,-1,0,1,3\}$.
\[
\begin{array}{c c c c c c c c c c c c c c}
&0 & 0 & 0 & 0 & 0 & 0 & 0 & 1 & 0 & 0 & 0 & 0 & 0\\
&& 0 & 0 & 0 & 0 & 1 & 0 & -1 & 0 & 0 & 0 & 1\\
&& & 1 & 0 & 0 &-1 & 0 & 1 & 0 & 0 & 0\\
&& & & 1 & 0 & 0 & 0 &-1 & 0 & 1\\
\textit{labels: } & \mathit{-3} & \mathit{-2} & \mathit{-1} & \mathit{0} & \mathit{1} & \mathit{2} & \mathit{3} & \mathit{4} & \mathit{5} & \mathit{6} & \mathit{7} & \mathit{8}& \mathit{9}
\end{array}
\]
Analogously to ASTs, we define $\M(A):=\M(\S(A))$ for an $(n,l)$-AS-trapezoid $A$.
For a centred Catalan set $S$ of size $n+1$, we define $w_{l}(S)$ (resp. $w_{l}(M)$) to be the number of $(n,l)$-trapezoids $A$ with $\S(A)=S$ (resp. $\M(A)=M$).

\begin{rem}
\label{rem: linking ASTs and ASts}
It is easy to see that the bijection between ASTs of order $n+1$ and $(n,2)$-AS-trapezoids commutes with the maps $\S$ and $\M$, which map ASTs and AS-trapezoids to centred Catalan sets or Motzkin paths, respectively.
This implies
\begin{align*}
w_2(S)&=w(S),\\
w_2(M)&=w(M).
\end{align*}
Since our interest in ASTs in this paper is with respect to their centred Catalan set and Motzkin path refinement, we can treat ASTs without loss of generality as a special case of AS-trapezoids.
\end{rem}

By using a computer system to calculate the number of $(n,l)$-AS-trapezoids for $1 \leq n\leq 9$ and treating $l$ as a variable, it is possible to guess the formula that enumerates $(n,l)$-AS-trapezoids. The formula was first conjectured in a more general way by Behrend \cite{ASt-conj} and later independently by the author. Behrend and Fischer \cite{ASt-conj} found a proof using the six-vertex model and Fischer \cite{Fischer1804.08681} additionally found a proof based on her operator formula for monotone triangles.

\begin{thm}[{\cite{ASt-conj,Fischer1804.08681}}]
\label{thm: AS-ts}
The number of $(n,l)$-AS-trapezoids is
\begin{multline}
\label{eq: number of ASts}
2^{\lfloor\frac{n+1}{2}\rfloor \lfloor\frac{n+2}{2}\rfloor-\lfloor\frac{n}{2}\rfloor} \prod_{i=1}^{\lfloor\frac{n+1}{2}\rfloor}\frac{(i-1)!}{(n-i)!}\\
\times \prod_{i \geq 0}
\left(\frac{l}{2}+3i+2\right)_{\lfloor\frac{n-4i-1}{2}\rfloor}
 \left(\frac{l}{2}+3i+2\right)_{\lfloor\frac{n-4i-2}{2}\rfloor}\\ 
 \times \prod_{i \geq 0}
\left(\frac{l}{2}+2\lfloor\frac{n}{2}\rfloor-i+\frac{1}{2}\right)_{\lfloor\frac{n-4i}{2}\rfloor}
 \left(\frac{l}{2}+2\lfloor\frac{n-1}{2}\rfloor-i+\frac{3}{2}\right)_{\lfloor\frac{n-4i-3}{2}\rfloor},
\end{multline}
where $(x)_n$ denotes the Pochhammer symbol $(x)_n:= x(x+1)\ldots (x+n-1)$ for non-negative integers $n$ and $(x)_n:=1$ if $n$ is negative.
\end{thm}

Remarkably, the above formula appears in two further places. First, it is identical to the determinant
\begin{equation}
\label{eq: Det Andrews}
\det_{0 \leq i,j \leq n-1} \left( \binom{l+i+j}{j}+q \delta_{i,j} \right)
\end{equation}
for $q=1$. When $q=1$, this determinant appeared in \cite{Andrews79}, where it was used to count descending plane partitions with parts of size less than $n$ by setting $l=2$. In \cite{CiucuEisenkoelblKrattenthalerZare01}, it was shown that for $q$ a sixth root of unity, this determinant gives the weighted enumeration of cyclically symmetric lozenge tilings of a hexagon with side lengths $n,n+l,n,n+l,n,n+l$ with a central triangular hole of side length $l$.
Second, the above formula is the $(-q)^{-n}$-th multiple of
\begin{equation}
\label{eq: Det ASMs}
\det_{0 \leq i,j \leq n-1} \left( \binom{l+i+j}{j}\frac{1-(-q)^{j+k-i}}{1+q}\right),
\end{equation}
where $k=3$ and $q$ is a primitive third root of unity. This determinant appeared first in \cite{Fischer18} in the special case $k=1$ and for $k=1$ and $l=0$, it enumerates the number of ASMs of size $n$. In \cite{Aigner_Det}, it is proven that for $k=1$ and $l=0$ this determinant is the $(q^{-1}+2+q)$-enumeration of ASMs, i.e., a weighted enumeration of ASMs where each ASM is weighted by $(q^{-1}+2+q)$ to the power of the number of $-1$s in the ASM. Evaluating the determinant at a primitive third, fourth, or sixth root of unity yields an alternative proof of the ASM Theorem and the $2$- and $3$-enumeration of ASMs based on the operator formula. We will study possible interpretation of $l$ in \eqref{eq: Det ASMs} and possible connections between \eqref{eq: number of ASts}, \eqref{eq: Det Andrews} and \eqref{eq: Det ASMs} in a forthcoming work.

%%%%%%%%%%%%%%%%%%%%%%%%%%%%%%%%%%%%%%%%%%%%%%%%%%%%%%%%%%%%%%%%%%%%%%%%%%%%%%%%%%%%%%%%%%%%%%%%%%%%%%
%%%%%%%%%%%%%%%%%%%%%%%%%%%%%%%%%%% The structure of AS-trapezoids %%%%%%%%%%%%%%%%%%%%%%%%%%%%%%%%%%%
%%%%%%%%%%%%%%%%%%%%%%%%%%%%%%%%%%%%%%%%%%%%%%%%%%%%%%%%%%%%%%%%%%%%%%%%%%%%%%%%%%%%%%%%%%%%%%%%%%%%%%

\section{The structure of AS-trapezoids}
\label{sec: structure of AS-ts}

In the following we prove two technical lemmas. Together they immediately imply Proposition \ref{prop: multiplicity of weights}.

\begin{lem}
\label{lem: structure of ones}
Let $S$ be a centred Catalan set of size $n+1$, $A$ an $(n,l)$-AS-trapezoid with $\S(A)=S$ and $\M(S)=(m_1, \ldots ,m_{n})$. Then the number of $1$s in the $i$-th row of $A$ is at most $1+\sum_{j=1}^i m_j$. Further, for $l \geq 2$ there exists an $(n,l)$-AS-trapezoid such that these bounds are sharp.
\end{lem}

\begin{proof}
We define an \emph{allowed position} for a $1$ in the $i$-th row of $A$ as a position such that the next non-zero entry below is negative or all entries below are $0$ and the label of the column corresponds to an element in $S$. Denote by $a_i$ the number of allowed positions for a $1$ in the $i$-th row.
By definition of $\M(S)$ we have $m_1= |\{-1,1\}\cap S|-1$. This implies that $m_1+1$ entries of the first row are in a column corresponding to an element in $S$, i.e., we obtain $a_1=m_1+1$. Since every $1$ (resp. $-1$) in the $i$-th row cancels out (resp. adds) an allowed position for a $1$ in the $(i+1)$-st row and there is one more $1$ than $-1$ in every row, the number of allowed positions in the central $l+2i-1$ columns of the $(i+1)$-st row is $a_i-1$. There are two new columns in row $i+1$ of which $|S\cap\{-i-1,i+1\}|=m_{i+1}+1$ have a label corresponding to an element in $S$. Hence we obtain $a_{i+1}=a_i+m_{i+1}$ and therefore
\[
 a_{i}=a_1+\sum_{j=2}^{i} m_i=1+\sum_{j=1}^{i} m_i.
\]
We construct an $(n,l)$-AS-trapezoid $A^\prime$ with $\S(A^\prime)=S$ and a maximal number of $1$s in a recursive manner. Place a $1$ in all allowed positions in the $i$-th row and put a $-1$ between the $1$ entries. This implies that in all columns the non-zero entries alternate. Since the allowed positions are either the leftmost or rightmost positions of a row or above a $-1$ from the row before, two allowed positions are by induction not direct neighbours. Therefore, the non-zero entries in each row are alternating. By the above formula, there is only one allowed position in the top row. If there exists a column in $A^\prime$ with a $-1$ as first non-zero entry from the top there would be a second allowed position in the top row, hence this cannot happen. An allowed position can appear in the central $l-1$ columns only above a $-1$. Hence the column sum is zero for the central $l-1$ columns and the resulting array is an $(n,l)$-AS-trapezoid. 
\end{proof}

\begin{lem}
\label{lem: splitting lemma}
Let $S_1, S_2$ be centred Catalan sets, $A_1$ an $(|S_1|-1,l)$-AS-trapezoid and $A_2$ an $(|S_2|-1,l+2|S_1|-2)$-AS-trapezoid with $\S(A_i)=S_i$ for $i=1,2$. By placing $A_2$ centred above $A_1$, we obtain an $(|S_1 \circ S_2|-1,l)$-AS-trapezoid with centred Catalan set $S_1 \circ S_2$. Moreover, every $(|S_1\circ S_2|-1,l)$-AS-trapezoid $A$ with $\S(A)=S_1 \circ S_2$ is of the above form.
\end{lem}

\begin{proof}
It follows immediately from the definitions of AS-trapezoids that this construction yields an $(|S_1|+|S_2|-2,l)$-AS-trapezoid with centred Catalan set $S_1 \circ S_2$.

On the other hand, let $A$ be an $(|S_1 \circ S_2|-1,l)$-AS-trapezoid with $\S(A)=S_1 \circ S_2$. We split $A$ into a bottom part $A_1$ consisting of the first $|S_1|-1$ rows from the bottom and a top part $A_2$ consisting of the remaining rows.
By Lemma \ref{lem: structure of ones}, there is only one allowed position for a $1$ in the top row of $A_1$.
If $A_1$ has a column whose first non-zero entry from the top is negative, there would be a second allowed position for a $1$ in the top row of $A_1$, hence this cannot happen. The central $l-1$ columns of $A$ have column sum zero and the first non-zero entry from top is positive, hence the partial column sum of the top $|S_2|-1$ rows is either $0$ or $1$. Since the first non-zero entry of every column of $A_1$ is also positive, the column sums of the central $l-1$ columns of $A_1$ also have column sum $0$ or $1$. Together with the above, this implies that the central $l-1$ columns of $A_1$ have column sum zero and therefore $A_1$ is an $(|S_1|-1,l)$-AS-trapezoid with $\S(A_1)=S_1$.
One of the central $l+2|S_1|-3$ columns of $A$ has a positive column sum if its column label corresponds to an element in $S_1$, which implies that this column of $A_1$ has a positive column sum. Therefore, the column sums of the central $l+2|S_1|-3$ columns of $A_2$ are zero, which implies that $A_2$ is an $(|S_2|-1,l+2|S_1|-2)$-AS-trapezoid.
\end{proof}

\begin{proof}[{Proof of Proposition \ref{prop: multiplicity of weights}}]
The theorem is a direct consequence of the above lemma.
\end{proof}

%%%%%%%%%%%%%%%%%%%%%%%%%%%%%%%%%%%%%%%%%%%%%%%%%%%%%%%%%%%%%%%%%%%%%%%%%%%%%%%%%%%%%%%%%%%%%%%%%%%%%%
%%%%%%%%%%%%%%%%%%%%%%%%%%%%%5%% A refined enumeration of AS-trapezoids %%%%%%%%%%%%%%%%%%%%%%%%%%%%%%
%%%%%%%%%%%%%%%%%%%%%%%%%%%%%%%%%%%%%%%%%%%%%%%%%%%%%%%%%%%%%%%%%%%%%%%%%%%%%%%%%%%%%%%%%%%%%%%%%%%%%%

\section{A refined enumeration of AS-trapezoids}
\label{sec: refined enum of AS-ts}

\subsection{A different perspective on AS-trapezoids}

The aim of this section is to prove that the weight function $w_l(S)$ is a polynomial in $l$. We will need the following definition due to Fischer \cite{Fischer11,Fischer1804.03630}.

\begin{defi}
 Let $1 \leq u<v \leq n$, $\mathbf{s}=(s_1,\ldots,s_u)$ be a weakly decreasing sequence of non-negative integers, $\mathbf{t}=(t_{v},\ldots,t_n)$ a weakly increasing sequence of non-negative integers and $\mathbf{k}=(k_1 ,\ldots,k_n)$ an increasing sequence of integers.
An $\st$-tree is an array of integers of the following shape satisfying the properties listed below. We obtain the shape of an $\st$-tree by starting with a triangular array of $n$ centred rows where the $i$-th row from top has $i$ entries.  Then we delete the  bottom $s_i$ elements in the $i$-th north-east diagonal for  $1 \leq i \leq u$ and the bottom $t_j$ elements in the $j$-th south-east diagonal for $v \leq j \leq n$, see Figure \ref{fig: s,t_tree}. 
 \begin{itemize}
  \item The  entries are weakly increasing in the north-east and south-east directions, and strictly increasing in the east direction.
  \item For $1 \leq i \leq u$, the bottom entry of the $i$-th north-east diagonal is $k_i$.
  \item For $v \leq i \leq n$ the bottom entry of the $i$-th south-east diagonal is $k_i$.
  \item The entries in the bottom row are $k_{u+1},\ldots, k_{v-1}$.
 \end{itemize} 
\begin{figure}[h!]
 \centering
 \includegraphics{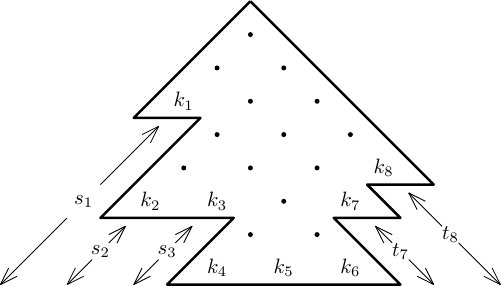}
  \caption{\label{fig: s,t_tree} Schematic diagram of an $\st$-tree.
  }
\end{figure}
\end{defi}

Let $f$ be a function in $x$. We define the following:
\begin{align*}
E_x (f)(x):=& f(x+1) \qquad &\textit{shift operator},\\
\overDelta{x}:=&E_x -\Id \qquad &\textit{forward difference},\\
\underDelta{x}:=& \Id - E_x^{-1} &\textit{backward difference}.
\end{align*}

\begin{thm}[{\cite[Section 4]{Fischer11}}]
\label{thm: s_t_trees}
 Set 
 \begin{equation}
  M_n(x_1,\ldots,x_n):= \prod_{1 \leq p < q \leq n}(\Id+ 
  \overDelta{x_p}\overDelta{x_q}+\overDelta{x_q})
  \prod_{1 \leq p < q \leq n} \frac{x_q-x_p}{q-p}.
 \end{equation}
 The number of $\st$-trees with bottom entries $\mathbf{k}=(k_1 ,\ldots,k_n)$ is given by
 \begin{equation}
  \label{eqn: s_t_trees}
  \left. \left(\prod_{i=1}^{u}(-\overDelta{x_i})^{s_i}
  \prod_{i=v}^n \underDelta{x_{i}}^{t_{i}}\right)
 M_n(x_1,\ldots,x_n)\right|_{\mathbf{x}=\mathbf{k}}.
 \end{equation}
\end{thm}

\begin{rem}
The above theorem is stated in \cite{Fischer11} in the more general setting where $\mathbf{k}$ is a weakly increasing sequence. The notion of $\st$-trees can be generalised to this setting using the notion of regular entries: An entry of an $\st$-tree is called \emph{regular} iff it has a south-east and a south-west neighbour. Then, the condition that the rows of an $\st$-tree are strictly increasing is replaced by the condition that two adjacent regular entries in the same row must be different.
\end{rem}

Let $S=\{s_1, \ldots, s_u,0,s_{u+1}, \ldots,s_{n}\}$ be a centred Catalan set of size $n+1$, $s_1 < \ldots <s_u \leq -1$ and $1 \leq s_{u+1}< \ldots < s_n$. Set $\mathbf{s}=(-s_1-1,\ldots, -s_{u}-1), \mathbf{t}=(s_{u+1}-1, \ldots,s_n-1)$ and $\mathbf{k}=(s_1+1,\ldots,s_u+1,l+s_{u+1}-1,\ldots,l+s_n-1)$.
The following algorithm is a bijection between $(n,l)$-AS-trapezoids $A$ with $\S(A)=S$ and $\st$-trees with bottom entries $\mathbf{k}$, for an example see Figure \ref{fig: ASt and (S,l)-trees}.
First we construct a triangular array $T_A$. We fill the $i$-th row from the bottom of $T_A$ with the column labels of $A$, for which the first non-zero entry above the $(i-1)$-st row is positive. In so doing, we write the numbers in an increasing order from left to right. The bottom row of $T_A$ is $\mathbf{k}$. Since there is one more $1$ than $-1$ in every row of the trapezoid, the number of entries in a row of $T_A$ is one less than the number of entries in the row below. Furthermore, it is easy to see that the entries of $T_A$ are weakly increasing in north-east and south-east direction.
For $1 \leq i \leq u$, the column sum of the $(s_i+1)$-st column in $A$ is equal to $1$. Since the first entry of $A$ in this column is in row $(-s_i)$, the first $(-s_i)$ entries of the $i$-th north-east diagonal of $T_A$ are equal to $s_i+1$. Hence we can delete $(-s_i-1)$ of these entries without losing information. Similarly, for $u+1 \leq i \leq n$ the first $s_i$ entries of the $i$-th south-east diagonal of $T_A$ will be $(l+s_i-1)$ and we can delete $(s_i-1)$ of them. The resulting array $T_A$ is an $\st$-tree with bottom entries $\mathbf{k}$.
Then, it is not difficult to see that every such $\st$-tree with bottom entries $\mathbf{k}$ is of the form $T_A$.
This proves the following proposition.

\begin{figure}
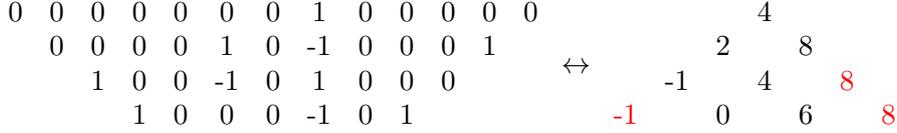

\begin{tabular}{c c c c c c c c c c c c c}
0 & 0 & 0 & 0 & 0 & 0 & 0 & 1 & 0 & 0 & 0 & 0 & 0\\
& 0 & 0 & 0 & 0 & 1 & 0 & -1 & 0 & 0 & 0 & 1\\
& & 1 & 0 & 0 &-1 & 0 & 1 & 0 & 0 & 0\\
& & & 1 & 0 & 0 & 0 &-1 & 0 & 1
\end{tabular}
$\leftrightarrow$
\begin{tabular}{c c c c c c c}
&&& 4\\
 && 2 && 8\\
 & -1 && 4 && \red{8}\\
\red{-1} && 0 && 6 && \red{8} 
\end{tabular}
\caption{\label{fig: ASt and (S,l)-trees} On the left is a $(4,6)$-AS-trapezoid and on the right the corresponding $\st$-tree. The red entries are not part of the $\st$-tree, but appear in its construction.}
\end{figure}

\begin{prop}
\label{prop: bij s,t-trees}
 Let $S$ be a centred Catalan set of size $n+1$. The map $A \mapsto T_A$ from $(n,l)$-AS-trapezoids with associated centred Catalan set $S$ to $(\mathbf{s},\mathbf{t})$-trees with bottom entries $\mathbf{k}$ described above is a bijection.
\end{prop}

An $\st$-tree with bottom entries $\mathbf{k}$ corresponding to an $(n,l)$-AS-trapezoids with associated centred Catalan set $S$ will be be called an \emph{$(S,l)$-tree}.

\subsection{Equivalence classes of $(S,l)$-trees}
In the following, we introduce an equivalence relation for AS-trapezoids and translate it to an equivalence relation for $(S,l)$-trees. The equivalence relation is more intuitively accessible in the AS-trapezoid picture, however, it is more convenient in the  $(S,l)$-tree setting.\\

\begin{defi}
Let $n,l$ be positive integers and $A,B$ two $(n,l)$-AS-trapezoids. For the central $l-1$ columns of $A$ (resp. $B$), denote from left to right the columns which have non-zero entries by $c_1(A), \ldots, c_f(A)$ (resp. $c_1(B), \ldots,$ $c_g(B))$. We say that $A$ equivalent to $B$ iff the following holds.
\begin{itemize}
\item The columns with labels less than or equal to $0$ are identical for $A$ and $B$.
\item The columns with labels greater than or equal to $l$ are identical for $A$ and $B$.
\item For all $1 \leq i \leq f$, $f=g$ and $c_i(A)=c_i(B)$.
\end{itemize}
We call the columns $c_1(A), \ldots, c_f(A)$ the free columns of $A$.
\end{defi}

\begin{ex}
\label{ex: equi of AS-trapezoids}
Let $A,B,C$ be the $(4,4)$-AS-trapezoids defined as below.
The columns with labels less than or equal to $0$ and the columns with labels greater than or equal to $l=4$ are identical for $A,B$ and $C$. The free columns are marked in red and green. 
\begin{align*}
A&= \, \begin{array}{c c c c c c c c c c c}
 0 & 0 & 0 & 0 & 0 & \red{0} & \green{1} & 0 & 0 & 0 & 0\\
  & 0 & 0 & 0 & 0 & \red{1} & \green{0} & 0 & 0 & 0\\
     && 1 & 0 & 0 & \red{0} & \green{-1}& 0 & 1\\
        &&& 1 & 0 &\red{-1} & \green{0} & 1
\end{array},\\
\\
B&= \, \begin{array}{c c c c c c c c c c c}
 0 & 0 & 0 & 0 & 0 & \red{1} & \green{0} & 0 & 0 & 0 & 0\\
  & 0 & 0 & 0 & 0 & \red{0} & \green{1} & 0 & 0 & 0\\
     && 1 & 0 & 0 & \red{-1} & \green{0}& 0 & 1\\
        &&& 1 & 0 &\red{0} & \green{-1} & 1
\end{array},\\
\\
C&= \, \begin{array}{c c c c c c c c c c c}
 0 & 0 & 0 & 0 & \red{0} & 0  & \green{1} & 0 & 0 & 0 & 0\\
  & 0 & 0 & 0 & \red{1} & 0  & \green{0} & 0 & 0 & 0\\
     && 1 & 0 & \red{0} & 0  & \green{-1}& 0 & 1\\
        &&& 1  &\red{-1} & 0  & \green{0} & 1
\end{array}.
\end{align*}
The free columns of $A$ and $C$ are
\[
c_1=\begin{matrix}
0 \\ 1 \\ 0 \\ -1
\end{matrix}, \qquad
c_2= \begin{matrix}
1 \\ 0 \\ -1 \\ 0
\end{matrix},
\]
hence $A$ and $C$ are equivalent.
The free columns of $B$ are $c_2,c_1$. The order of these columns is the opposite of the free columns of $A$, so the AS-trapezoids $A$ and $B$ are not equivalent.
\end{ex}

Let $A,B$ be $(n,l)$-AS-trapezoids. If $A$ and $B$ are equivalent, then $\S(A)=\S(B)$. This follows because the centred Catalan set $\S(A)$ depends only on the columns of $A$ with labels less than or equal to $0$ and columns with labels greater than or equal $l$, which are by definition fixed by the equivalence relation. Hence the equivalence relation on $(n,l)$-AS-trapezoids induces an equivalence relation for $(S,l)$-trees, which is given as follows.

\begin{defi}
Let $T=(T_{i,j}),T^\prime=(T^\prime_{i,j})$ be two $(S,l)$-trees. We call $T$ and $T^\prime$ \emph{equivalent} iff the following holds.
\begin{enumerate}
 \item For all $i,i^\prime,j,j^\prime$, $T_{i,j}< T_{i^\prime,j^\prime}$ holds if and only if $T^\prime_{i,j}< T^\prime_{i^\prime,j^\prime}$.
 \item For all $i,j$, if $T_{i,j} \leq 0$ or $T_{i,j} \geq l$, or $T^\prime_{i,j} \leq 0$ or $T^\prime_{i,j} \geq l$, then $T_{i,j}=T^\prime_{i,j}$.
\end{enumerate}
We define the \emph{number of free columns} of an $(S,l)$-tree $T$ as $|\{T_{i,j}: 0 < T_{i,j} <l\}|$. It follows from the definition that the number of free columns is constant on equivalence classes.
\end{defi}

\begin{ex}
\label{ex: equi-classes of trees}
Let $A$ be the first $(4,4)$-AS-trapezoid in Example \ref{ex: equi of AS-trapezoids}. The $(\{-2,-1,0,1,2\},4)$-tree $T_A$ corresponding to $A$ and the equivalence class $\overline{T_A}$ of $T_A$ are, respectively:
\[
T_A=\begin{array}{c c c c c}
&& 3 \\
& 2 && 3 \\
-1 && 2 && 5 \\
&0 && 4
\end{array}
\;\;\;and \qquad
\overline{T_A}=\begin{array}{c c c c c}
&& b \\
& a && b \\
-1 && a && 5 \\
&0 && 4
\end{array},
\]
with $0 < a < b < 4$. The number of free columns of $T_A$ is $2$.
\end{ex}

\begin{lem}
 \label{lem: size of an equivalence class}
 Let $S$ be a centred Catalan set, $l>0$ and $T$ an $(S,l)$-tree with $f$ free columns. Then the size of the equivalence class $\overline T$ of $T$ is given by
 \[
  |\overline T|= \binom{l-1}{f}.
 \]
\end{lem}
\begin{proof}
 Let $x_1,\ldots,x_f$ be entries in $T$ such that $0<x_1< \ldots < x_f <l$. Every element of the equivalence class $\overline T$ is uniquely described by the sequence of these values $(x_1, \ldots, x_f)$. Hence
 \[
  |\overline T| = |\{(x_1,\ldots,x_f): 0<x_1< \ldots < x_f <l \}|= \binom{l-1}{f}.\qedhere
 \] 
\end{proof}

Let $S=\{s_1,\ldots,s_u,0,s_{u+1},\ldots,s_n\}$ be a centred Catalan set of size $n+1$ with $s_1< \ldots <s_u< 0< s_{u+1}< \ldots < s_n$. We define two Young diagrams $\lambda(S), \mu(S)$ via
\begin{align*}
\lambda(S)&= (u+1-s_{u+1},\ldots, n-s_n)^\prime,\\
\mu(S)&= (-s_1-u,\ldots, -s_u-1),
\end{align*}
where $\lambda^\prime$ denotes the conjugate Young diagram, see Figure \ref{fig: skew shape}.
Alternatively, one can describe $\lambda(S)$ and $\mu(S)$ as the smallest Young diagrams such that $\lambda(S)/ \mu(S)$ is the area enclosed between the paths $\tilde{\lambda}(S)=(\tilde{\lambda}_i)_{1 \leq i \leq n}$ and $\tilde{\mu}(S)=(\tilde{\mu}_i)_{1 \leq i \leq n}$, which are defined in the following table.
\begin{center}
 \begin{tabular}{c || c |c}
 & $\tilde{\lambda}_i$ & $\tilde{\mu}_i$\\
 \hline
 $\{-i,i\} \subseteq S$ &E &N\\
 $-i \in S, i \notin S$ &N &N\\
 $i \in S, -i \notin S$ &E &E\\
 $-i,i \notin S$ &N &E\\
 \end{tabular}
\end{center}
Using this description, it is easy to see that $\mu(S)$ is included in $\lambda(S)$ and that the following holds
\begin{equation}
\label{eq: areas are equal}
 |\lambda(S)/ \mu(S)| = \ar(\M(S)).
\end{equation}

\begin{rem}
\label{rem: diagrams and concatenation}
Let $S_1,S_2$ be two centred Catalan sets, each of the form $S_i=\{s_{i,1},\ldots, s_{i,u_i},0,s_{i,u_i+1},\ldots,s_{i,n_i}\}$ with $s_{i,1} < \ldots < s_{i,u_i} < 0 < s_{i,u_i+1} < \ldots < s_{i,n_i}$ and $\lambda(S_i)= (\lambda_{i,j})_{1 \leq j \leq u_i}$ and $\mu(S_i)=(\mu_{i,j})_{1 \leq j \leq u_i}$ for $1 \leq i \leq 2$. It follows from the definition of the Young diagrams $\lambda,\mu$ that the diagrams corresponding to the concatenation $S_1 \circ S_2$ are given by
\begin{align*}
\lambda(S_1 \circ S_2) = (\lambda_{2,1}+u_1,\ldots, \lambda_{2,u_2}+u_1,\lambda_{1,1}, \ldots, \lambda_{1,u_1}),\\
\mu(S_1 \circ S_2) = (\mu_{2,1}+u_1,\ldots, \mu_{2,u_2}+u_1,\mu_{1,1}, \ldots, \mu_{1,u_1}).
\end{align*}
For an example, see Figure \ref{fig: skew shape}.
Hence the size of the skew-shaped Young diagram is given by
\[
|\lambda(S_1 \circ S_2)/\mu(S_1 \circ S_2)|=|\lambda(S_1)/\mu(S_1)|+|\lambda(S_2)/\mu(S_2)|.
\]

\begin{figure}
\centering
\includegraphics[width=\textwidth]{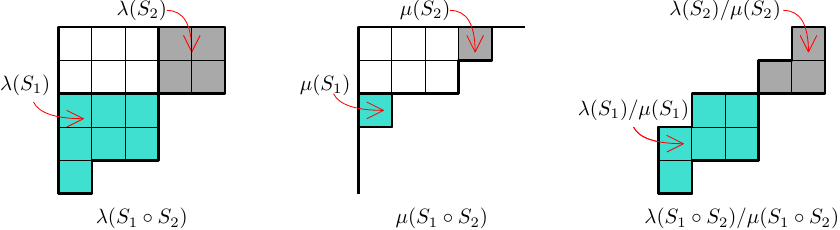}
\caption{\label{fig: skew shape} The (skew-shaped) Young diagrams corresponding to $S_1=\{ -4,-2,-1,0,1,3,4\},S_2=\{ -3,-1,0,1,2\}$ and $S_1 \circ S_2$.}
\end{figure}
\end{rem}

\begin{lem}
 \label{lem: number of equivalence classes}
 Let $S$ be an irreducible centred Catalan set of size $n$, then the following are true.
\begin{enumerate}
\item The number of free columns of an $(S,l)$-tree is at most $|\lambda(S)/ \mu(S)|$.
\item For $l > |\lambda(S)/ \mu(S)|$, the equivalence classes of $(S,l)$-trees are in bijection with the equivalence classes of $(S, |\lambda(S)/ \mu(S)|+1)$-trees.
\item For $l > |\lambda(S)/ \mu(S)|$ the number of equivalence classes with  maximal free columns is equal to $f^{\lambda(S)/ \mu(S)}$, the number of standard Young tableaux of skew shape $\lambda(S)/ \mu(S)$.
\end{enumerate}
\end{lem}

\begin{proof}
 Let $T$ be an $(S,l)$-tree. We draw a box around all entries of $T$ that are not fixed by definition. (An example is on the left side of Figure \ref{fig: from ASt to SYT}.) Denote by $b_i$ the number of boxes in the $i$-th row from the bottom of $T$. The definition of $(S,l)$-trees implies that
 \begin{itemize}
 \item there is no box in the bottom row,
 \item there is a box north-west of the leftmost box in row $i$ iff $-i \in S$,
 \item there is a box north-east of the rightmost box in row $i$ iff $i \in S$.
 \end{itemize}
Hence, the outer shape is determined by the two paths $\lambda^\prime$ and $\mu^\prime$, where the $i$-th steps $\lambda^\prime_i,\mu^\prime_i$ of $\lambda^\prime, \mu^\prime$ are given by
 \begin{center}
 \begin{tabular}{c || c |c}
 & $\lambda^\prime_i$ & $\mu^\prime_i$\\
 \hline
 $\{-i,i\} \subseteq S$ &NE &NW\\
 $-i \in S, i \notin S$ &NW &NW\\
 $i \in S, -i \notin S$ &NE &NE\\
 $-i,i \notin S$ &NW &NE\\
 \end{tabular}
\end{center}
By rotating the box complex by $45^\circ$ clockwise, we obtain the skew-shaped diagram $\lambda(S)/\mu(S)$ whose entries are the entries of $T$ not fixed by definition. Hence the maximal number of free columns is $|\lambda(S)/ \mu(S)|$.\\

Let $l>|\lambda(S)/ \mu(S)|$ and $\overline{T}$ be an equivalence class of $(S,l)$-trees with $f$ free columns. Then there exists a representative $T$ such that the entries that lie in a box have values between $1$ and $f$. Denote by $T^\prime$ the $(S,|\lambda(S)/ \mu(S)|+1)$-tree which is obtained from $T$ by subtracting $l-(|\lambda(S)/ \mu(S)|+1)$ from all positive entries in $T$ which do not lie in a box. The map $\overline{T} \mapsto \overline{T^\prime}$ is then a bijection of equivalence classes of $(S,l)$-trees to equivalence classes of $(S,|\lambda(S)/ \mu(S)|+1)$-trees.\\

The above argument implies that we can restrict ourselves to $l=$ \linebreak$|\lambda(S)/ \mu(S)|+1$.  For an equivalence class with a maximal number of free columns, all entries in the boxes are different. Since $l=|\lambda(S)/ \mu(S)|+1$, the entries in the boxes are the integers $1, \ldots, |\lambda(S)/ \mu(S)|$. The entries in an $\st$-tree are strictly increasing along rows from left to right and weakly increasing along north-east and south-east diagonals from bottom to top. Therefore, the filling of the skew-shaped Young diagram is strictly increasing along columns and rows, i.e., it is a skew-shaped standard Young tableau. Hence every equivalence class with maximal free columns corresponds to a standard Young tableau of skew shape $\lambda(S)/\mu(S)$. It is obvious that the converse is also true.
\end{proof}

  \begin{figure}
 \centering
 \includegraphics{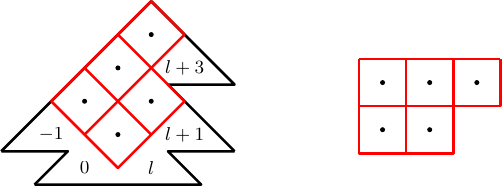}
 \caption{\label{fig: from ASt to SYT} Schematic representation of an $(S,l)$-tree and the skew-shaped Young diagram $\lambda(S)/\mu(S)$ for $S=\{-2,-1,0,1,2,-4\}$.}
\end{figure}

We can now prove our second main result.

\begin{proof}[{Proof of Theorem \ref{thm: Catalan weight is polynomial}}]
 By Proposition \ref{prop: multiplicity of weights} and Remark \ref{rem: diagrams and concatenation}, it suffices to assume that $S$ is irreducible. As a consequence of Theorem \ref{thm: s_t_trees} and Proposition \ref{prop: bij s,t-trees}, the number $w_l(S)$ of $(S,l)$-trees is a polynomial in $l$. 
Let $d(S)$ be the degree of $w_l(S)$.  We show by induction on $\ar(\M(S))$ that $d(S) \leq \ar(\M(S))$. Write $S$ in the form $S=\{s_1, \ldots,s_u, 0, s_{u+1}, \ldots, s_n\}$ with $s_1<\ldots<s_u<0<s_{u+1}< \ldots < s_n$. The degree $d(S)$ is at most the degree of the polynomial $M_n(x_1,\ldots,x_n)$ minus the number of $\overline{\Delta}, \underline{\Delta}$ operators appearing in \eqref{eqn: s_t_trees} and therefore
  \begin{equation}
  \label{eqn: degree inequality}
  d(S) \leq \binom{n}{2} - \sum_{i=1}^u (-s_i-1) -\sum_{i=u+1}^n (s_i-1) =
  \binom{n}{2}- \sum_{i \in S\setminus\{0\}}(|i|-1).
  \end{equation}
We claim that
\begin{equation}
\label{eq: to show I}
\binom{n}{2}=\sum_{i \in S\setminus\{0\}}(|i|-1) + \ar(\M(S)),
\end{equation}
which, together with \eqref{eqn: degree inequality}, will imply that $d(S) \leq \ar(\M(S))$.
The area $\ar(\M(S))$ is equal to $0$ iff for all $1 \leq i \leq n$ either $i \in S$ or $-i \in S$, hence
 \[
 \sum_{i \in S\setminus\{0\}}(|i|-1)= \sum_{i=1}^n (i-1)=\binom{n}{2}.
 \]
 If $\ar(\M(S))>0$, denote by $i_0$ the largest integer with $1 \leq i_0 \leq n-1$ and $\{-i_0,i_0 \} \subseteq S$. Then the centred Catalan set
 \[
 S^\prime:=\begin{cases}
 (S\setminus\{i_0\}) \cup \{ i_0+1\} & (i_0+1) \notin S,\\
 (S\setminus\{i_0\}) \cup \{ -(i_0+1)\} & -(i_0+1) \notin S,\\
 \end{cases}
 \]
  is well defined. 
The paths $\M(S)$ and $\M(S^\prime)$ differ only in the $i_0$-th and $(i_0+1)$-st step, as shown in Figure \ref{fig: change of S}. This implies that $\ar(\M(S))=\ar(\M(S^\prime))+1$. On the other hand, the sum over all $i \in S\setminus\{0\}$ in \eqref{eq: to show I} is one less than the sum over all $i \in S^\prime\setminus \{0\}$, which proves \eqref{eq: to show I}.\\
 \begin{figure}
 \centering
 \includegraphics{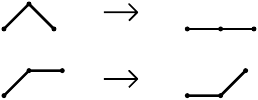}
 \caption{\label{fig: change of S} The two possibilities for changes between $M(S)$ and $M(S^\prime)$ at the positions $i_0,i_0+1,i_0+2$.}
\end{figure}

By Lemma \ref{lem: size of an equivalence class}, an equivalence class of $(S,l)$-trees with $f$ free columns contributes $\binom{l-1}{f}$ to $w_l(S)$. Lemma \ref{lem: number of equivalence classes} states that there are $f^{\lambda(S)/ \mu(S)}$ equivalence classes with $|\lambda(S)/\mu(S)|=\ar(\M(S))$ free columns. Hence $w_l(S)$ has degree $\ar(\M(S))$ and the leading coefficient is $\frac{f^{\lambda(S)/ \mu(S)}}{|\lambda(S)/\mu(S)|!}$.
\end{proof}

Since $w_l(M)=\sum_{S: \M(S)=M}w_l(S)$, the following is a direct consequence of the above theorem.

\begin{cor}
Let $M=M_1 \circ \ldots \circ M_k$ be a Motzkin path and $M_1,\ldots, M_k$ its irreducible components. The weight function $w_l(M)$ is a polynomial in $l$ of degree $\ar(M)$ with leading coefficient
\begin{equation}
\label{eq: leading cof motzkin path}
\prod_{i=1}^k\sum_{S: \M(S)=M_i}\frac{f^{\lambda(S)/\mu(S)}}{\ar(M_i)!}.
\end{equation}
\end{cor}

It is unknown to the author if \eqref{eq: leading cof motzkin path} can be simplified.

\begin{rem}
\label{rem: ASTs and FPLs}
For a centred Catalan set of size $n$, denote by $\pi(S)$ the non-crossing matching of size $n$ corresponding to $\D(S)$. Let  $S_1,S_2$ be two irreducible centred Catalan sets of size $n_1$ and $n_2$, respectively, and let $m$ be a positive integer. We define $S(m)=\{0,1,\ldots, m \} \cup \s_{m}(S_1) \cup \{n_1+m, n_1+m+1,\ldots,n_1+2m+1\} \cup \s_{n_1+2m+1}(S_2)$. The non-crossing matching $\pi(S(m))$ consists of $\pi(S_1)$ enclosed by $m$ `small arches' on every side concatenated with $\pi(S_2)$, see Figure \ref{fig: rem}. By Proposition \ref{prop: multiplicity of weights} and the fact that $w(\{0,1,\ldots,m\})=1$, the weight function $w(S(m))$	 can be written as
\begin{multline*}
w(S(m))=w_2(S(m))\\
=w_{2}(\{0,1,\ldots m\})w_{2(m+1)}(S_1)w_{2(n_1+m+1)}(\{0,1,\ldots m+1\})w_{2(n_1+2m+3)}(S_2)\\
=w_{2m}(S_1)w_{2(n_1+2m+3)}(S_2),
\end{multline*}
Hence Theorem \ref{thm: Catalan weight is polynomial} implies that $w(S(m))$ is a polynomial in $m$ of degree $|\lambda(S_1)/\mu(S_1)|+|\lambda(S_2)/\mu(S_2)|$ with leading coefficient
\begin{equation}
\label{eq: leading coef Sm}
\frac{f^{\lambda(S_1)/ \mu(S_1)} \, f^{\lambda(S_2)/ \mu(S_2)}}{|\lambda(S_1)/\mu(S_1)|!\cdot |\lambda(S_2)/\mu(S_2)|!}.
\end{equation}

\begin{figure}
 \centering
 \includegraphics[width=.7\textwidth]{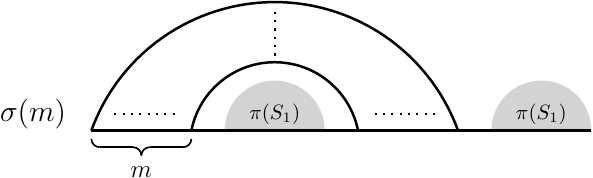}
 \includegraphics[width=.7\textwidth]{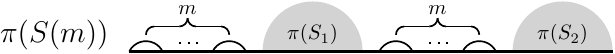}
  \caption{\label{fig: rem} Graphical representation of the non-crossing matchings $\sigma(m)$ in the top and $\pi(S(m))$ in the bottom.
  }
\end{figure}

An analogous theorem \cite[Theorem $1$]{Aigner_FPL}, \cite{CaselliKrattenthalerLassNadeau04} exists for fully packed loops (FPLs), objects that are equinumerous to ASTs. Denote by $\sigma(m):=$\linebreak$(\pi(S_1))_m\pi(S_2)$ the non-crossing matching which consists of $\pi(S_1)$ enclosed by $m$ `nested arches' concatenated with $\pi(S_2)$, see Figure \ref{fig: rem}. Then this theorem states that the number of FPLs associated to $\sigma(m)$ is a polynomial in $m$ of degree $|\lambda(\pi(S_1))|+|\lambda(\pi(S_2))|$ with leading coefficient
\begin{equation}
\label{eq: leading coef SigmaM}
\frac{f^{\lambda(\pi(S_1))} \, f^{\lambda(\pi(S_2))}}{|\lambda(\pi(S_1))|!\cdot |\lambda(\pi(S_2))|!},
\end{equation}
where $\lambda(\pi)$ is a Young diagram associated to the non-crossing matching $\pi$. 
This is remarkable for three reasons. First, the refinements of ASTs and FPLs by Catalan objects are of different nature, however we have in both cases a polynomiality theorem. Second, the degree and the leading coefficient of the polynomials are almost the same, with the important difference that the degree and the leading coefficient are linked to standard Young tableaux in the FPL case, whereas they are linked to skew-shaped standard Young tableaux in the AST case.
 Third, the two refinements seem to be in some sense `dual': for one, the polynomiality results differ by the dual notions of `nested arches' and `small arches'. Then, the number of FPLs linked to a non-crossing matching is minimal if it consists of nested arches and is maximal if it consists of small arches. Furthermore, computer experiments suggest that for an centred Catalan set $S$ the weight $w_l(S)$ is maximal
  if $\pi(S)$ consists of nested arches.
 \end{rem}

%%%%%%%%%%%%%%%%%%%%%%%%%%%%%%%%%%%%%%%%%%%%%%%%%%%%%%%%%%%%%%%%%%%%%%%%%%%%%%%%%%%%%%%%%%%%%%%%%%%%%%
%%%%%%%%%%%%%%%%%%%%%%%%%%%%%%%%5%%%%%% A constant term expression  %%%%%%%%%%%%%%%%%%%%%%%%%%%%%%%%%%
%%%%%%%%%%%%%%%%%%%%%%%%%%%%%%%%%%%%%%%%%%%%%%%%%%%%%%%%%%%%%%%%%%%%%%%%%%%%%%%%%%%%%%%%%%%%%%%%%%%%%%

\section{A constant term expression for AS-trapezoids}
\label{sec: const term}

Using Theorem \ref{thm: s_t_trees}, we will derive a constant term expression for $w_l(S)$. Our proof follows the same steps as the proof of Theorem 7 in \cite{Fischer1804.03630}.
The following identities will be needed later. The first is an easy consequence of the definitions of the operators $E_x$, $\overDelta{x}$, and $\underDelta{x}$, while the second is stated in \cite[Lemma 5]{Fischer07}.
\begin{equation}
\label{eq: operator}
\Id +\overDelta{y} + \overDelta{x} \overDelta{y}= E_x E_y(\Id -\underDelta{x}+\underDelta{x} \underDelta{y}) =E_y(\Id +\overDelta{x}\underDelta{y}),
\end{equation}
\begin{equation}
\label{eq: M rotation}
M_n(x_1,\ldots, x_n)= (-1)^{n-1}M_n(x_2,\ldots, x_n,x_1-n).
\end{equation}

\begin{proof}[{Proof of Theorem \ref{thm: constant term identity}}]
We set $\mathbf{s}=\{-s_1-1,\ldots, -s_u-1\},\mathbf{t}=\{s_{u+1}-1,\ldots, s_n-1\},\mathbf{k}=\{s_1+1, \ldots s_u+1, l+s_{u+1}-1, \ldots, l+s_n-1\}$. Since $(n,l)$-AS-trapezoids correspond to $\st$-trees with bottom entries $\mathbf{k}$, Theorem \ref{thm: s_t_trees} states
\begin{multline}
\label{eq: wl with operators}
w_l(S)=
\left.\left(\prod_{i=1}^u (-\overDelta{x_i})^{-s_i-1} \prod_{i=u+1}^n \underDelta{x_{i}}^{s_i-1}\right) M_{n}(x_1,\ldots, x_{n}) \right|_{\mathbf{x}=\mathbf{k}}\\
 =\left.\left(\prod_{i=1}^u (-\overDelta{x_i})^{-s_i-1}E_{x_i}^{s_i+1} \prod_{i=u+1}^n \underDelta{x_i}^{s_i-1}E_{x_i}^{l+s_i-1} \right) M_n(x_1,\ldots, x_n) \right|_{\mathbf{x}=\mathbf{0}}\\
= \left.\left(\prod_{i=1}^u (-\underDelta{x_i})^{-s_i-1}\prod_{i=u+1}^n \overDelta{x_i}^{s_i-1}E_{x_i}^{l} \right) M_n(x_1,\ldots, x_n) \right|_{\mathbf{x}=\mathbf{0}}.
\end{multline}
Using \eqref{eq: operator} and \eqref{eq: M rotation}, we can write $M_n(x_1,\ldots, x_n)$ as
\begin{multline}
\label{eq: close look at M}
 M_n(x_1,\ldots,x_n) = (-1)^{u(n-1)}M_n(x_{u+1},\ldots, x_n,x_1-n,\ldots,x_u-n)\\
 =(-1)^{u(n-1)} \prod_{i=1}^u E_{x_i}^{-n}
 \prod_{1 \leq i <j \leq u}E_{x_i}E_{x_j}(\Id - \underDelta{x_i}+\underDelta{x_i}\underDelta{x_j})\\
 \times \prod_{u+1 \leq i <j \leq n} (\Id+\overDelta{x_j}+\overDelta{x_i}\overDelta{x_j})
 \prod_{i=1}^u \prod_{j=u+1}^n E_{x_i}(\Id+\underDelta{x_i}\overDelta{x_j})
\\
 \times (-1)^{u(n-1)} \prod_{1 \leq i < j \leq n}\frac{x_j-x_i}{j-i}
 = \OP \prod_{i=1}^u  E_{x_i}^{-1}  \det_{1 \leq i,j \leq n}\left( \binom{x_i}{j-1} \right),
\end{multline}
with 
\begin{multline*}
\OP= \prod_{1 \leq i <j \leq u}(\Id - \underDelta{x_i}+\underDelta{x_i}\underDelta{x_j})\\
\times \prod_{u+1 \leq i <j \leq n} (\Id+\overDelta{x_j}+\overDelta{x_i}\overDelta{x_j})
\prod_{i=1}^u \prod_{j=u+1}^n (\Id+\underDelta{x_i}\overDelta{x_j}).
\end{multline*}
Using \eqref{eq: wl with operators} and \eqref{eq: close look at M}, we obtain the following expression for $w_l(S)$
\begin{multline}
\label{eq: wl umgeformt}
w_l(S)=  \OP 
 \left.\prod_{i=1}^u(-\underDelta{x_i})^{-s_i-1}E_{x_i}^{-1}
\prod_{i=u+1}^n \overDelta{x_i}^{s_i-1}E_{x_i}^{l} \det_{1 \leq i,j \leq n}\left( \binom{x_i}{j-1} \right)\right|_{\mathbf{x}=\mathbf{0}}\\
  =\left.(-1)^{u+\sum_{i=1}^u s_i}\OP \sum_{\sigma \in \Sgrp_n}\sgn(\sigma)
 \prod_{i=1}^u \binom{x_i+s_i}{\sigma(i)+s_i}\prod_{i=u+1}^n \binom{x_i+l}{\sigma(i)-s_i}
 \right|_{\mathbf{x}=\mathbf{0}}.
\end{multline}
Since $(-\underDelta{x}) \binom{x+s}{j+s}(-1)^s=E_s^{-1} \binom{x+s}{j+s}(-1)^s$ and $\overDelta{x} \binom{x+l}{j-s}=E_s \binom{x+l}{j-s}$, we can replace $\OP$ by $\OP^\prime$, where $\OP^\prime$ is $\OP$ with every $-\underDelta{x_i}$ replaced by $E_{s_i}^{-1}$ and every $\overDelta{x_j}$ replaced by $E_{s_j}$. As a consequence of this, we can now evaluate \eqref{eq: wl umgeformt} at $x_1,\ldots,x_n=0$ and obtain
\begin{align*}
w_l(S)=&
\OP^\prime(-1)^{u+\sum_{i=1}^u s_i} \sum_{\sigma \in \Sgrp_{n}}\sgn(\sigma)
\prod_{i=1}^u \binom{s_i}{\sigma(i)+s_i}\prod_{i=u+1}^n \binom{l}{\sigma(i)-s_i}.
\end{align*}
In order to evaluate this expression we use the following identity. Let $f:\mathbb{Z}^n \rightarrow \mathbb{R}$ be a function such that the generating function $F(x_1,\ldots,x_n):=\sum_{a \in \mathbb{Z}^n}f(a)x^a$ is a Laurent series and let $p(x_1,\ldots,x_n)$ be a Laurent polynomial. Then
\begin{align}
\label{eq: operator angewandt gleich ct}
p(E_{a_1},\ldots,E_{a_n})f(a)=\CT_{x_1,\ldots,x_n} \prod_{i=1}^nx_i^{-a_i}p(x_1^{-1},\ldots,x_n^{-1})F(x_1,\ldots,x_n),
\end{align}
where $\CT_{x_1,\ldots,x_n}$ denotes the constant term in $x_1,\ldots, x_n$.
This identity can be proven easily if $p$ is a monomial. The general statement follows immediately since $p$ is a linear combination of monomials. We apply \eqref{eq: operator angewandt gleich ct} for the function
\begin{multline*}
f(s_1,\ldots,s_n):= (-1)^{u+\sum_{i=1}^u s_i} \\
\times \sum_{\sigma \in \Sgrp_{n}}\sgn(\sigma)
\prod_{i=1}^u \binom{s_i}{\sigma(i)+s_i}\prod_{i=u+1}^n \binom{l}{\sigma(i)-s_i}.
\end{multline*}
By using
\begin{align*}
\sum_{s \in \mathbb{Z}}(-1)^s\binom{s}{\sigma+s}x^s &= (-x)^{-1}(-1-x^{-1})^{\sigma-1},\\
\sum_{s \in \mathbb{Z}}\binom{l}{\sigma-s}x^s &= x^{\sigma-l}(1+x)^{l}.
\end{align*}
we obtain for $F$
\begin{multline*}
F(x_1,\ldots,x_{n})= (-1)^{u}\sum_{\sigma \in \Sgrp_{n}}\sgn(\sigma)\\
\times \prod_{i=1}^u (-x_i)^{-1}(-1-x_i^{-1})^{\sigma(i)-1}\prod_{i=u+1}^{n}x_i^{\sigma(i)-l}(1+x_i)^{l}\\
=(-1)^{u}\det_{1 \leq i,j \leq n}\left(
\begin {cases}
 (-x_i)^{-1}(-1-x_i^{-1})^{j-1} \quad & i\leq u\\
x_i^{j-l}(1+x_i)^{l} & i >u
\end{cases}\right)\\
=\prod_{i=1}^u x_i^{-u}\prod_{i=u+1}^{n}x_i^{1-l}(1+x_i)^{l}
\prod_{\substack{1 \leq i<j\leq u, \\u+1 \leq i<j\leq n}}(x_j-x_i)\prod_{i=1}^u\prod_{j=u+1}^{n}(1+x_i^{-1}+x_j).
\end{multline*}
In the last step, we evaluated the determinant using the Vandermonde determinant evaluation.
Hence we obtain
\begin{multline*}
w_l(S)=\CT_{x_1,\ldots,x_{n}}\prod_{i=1}^ux_i^{-n-s_i}\\
\prod_{i=u+1}^n x_i^{-n-l-s_i+2}(1+x_i)^l
\prod_{1\leq i < j \leq n}(x_j-x_i)(1+x_i+x_ix_j). \qedhere
\end{multline*}
\end{proof}

\section{Roots of $w_l(S)$ and $w_l(M)$}
\label{sec: roots}

Studies of $w_l(S)$ and $w_l(M)$ for small centred Catalan sets $S$ and Motzkin paths $M$ show that $w_l(S)$ and $w_l(M)$ often have rational roots, where almost all of these are even integer roots. In Appendix \ref{app: Data}, we list  $w_l(S), w_l(M)$ for all irreducible centred Catalan sets up to size six and irreducible Motzkin paths up to length five. This section contains results and conjectures based on a first analysis of a slightly larger data set. The conjectures can be divided into three types.
\begin{itemize}
\item Conjecture \ref{conj: possible roots} predicts which rational numbers appear as roots of weight functions.
\item The conjectures in Section \ref{sec: roots of S} state relations between the rational roots of $w_l(S)$ and the beginning of the Motzkin path $\M(S)$.
\item In Section \ref{sec: roots of M}, a conjecture connects the rational roots of $w_l(M)$ to the ending of the Motzkin path $M$.
\end{itemize}
The ultimate goal is a description of all rational roots of $w_l(S)$ and $w_l(M)$ similar to \cite[Conjecture 1.2]{FonsecaNadeau11}.

\begin{conj}
\label{conj: possible roots}
\begin{enumerate}
\item Let $M$ be an irreducible Motzkin path of length $n\geq 8$. The rational roots of $w_l(M)$ lie in $\{-1,-2,\ldots,-2n+2\}$, and for every integer $x$ in this set there exists a Motzkin path $M$ such that $x$ is a root of $w_l(M)$.
\item Let $S$ be an irreducible centred Catalan set of size $n\geq 11$. The rational roots of $w_l(S)$ lie in 
$\{-1, -2,\ldots,-2n+4,-\frac{n^2-5n+7}{(n-3)}\}$, and for every rational number $x$ in this set there exists a centred Catalan set $S$ such that $x$ is a root of $w_l(S)$.
Furthermore, $-\frac{n^2-5n+7}{(n-3)}$ is a root of $w_l(S)$ iff  $S=\{-n+2,-1,0,1, \ldots, n-3\}$ or $S=\{-n+3, -n+4,\ldots, -1,0,1, n-2\}$.
\end{enumerate}
\end{conj}

\subsection{Rational roots of $w_l(S)$}
\label{sec: roots of S}

\begin{conj}
\label{conj: roots of cCs}
Let $S$ be a centred Catalan set. Then 
\[
\{-i,\ldots, i\} \subseteq S \quad \Leftrightarrow \quad 
\prod_{k=0}^{\lfloor\frac{i-1}{2}\rfloor}(l+1+3k)_{i-2k} \textnormal{ divides } w_l(S),
\]
where $(x)_k:=x(x+1)\cdots (x+k-1)$ is the Pochhammer symbol.
\end{conj}

For Motzkin paths, the above conjecture implies that the weight function has certain integer roots that depend on the number of consecutive north-east steps at the beginning of the path.
We can prove the above conjecture for $i=1,2$, however our proof technique does not extend to $i>2$.\\

\begin{prop}
 \label{prop: divisibility by 2j+1}
 Let $S$ be a centred Catalan set. Then the following holds:
 \begin{align*}
 \{-1,1\}\subseteq S &\Leftrightarrow  (l+1)|w_l(S),\\
 \{-2,-1,1,2\} \subseteq S& \Leftrightarrow  (l+1)(l+2)|w_l(S). 
\end{align*}  
\end{prop}

\begin{proof}
By Proposition \ref{prop: multiplicity of weights}, the weight function $w_l(S)$, where $S=S_1 \circ S_2$, factorises into $w_l(S) = w_l(S_1) w_{l+2|S_1|-2}(S_2)$. Hence, it suffices to prove the proposition for irreducible $S$.
The proof is based on the following fact. Let $p(x)$ be a polynomial in $x$ and define
\begin{equation}
\label{eq: sum of poly}
 P(x)=\begin{cases}
       \sum_{i=0}^x p(i) \quad &x \geq 0,\\
       0 &x=-1,\\
       -\sum_{i=x+1}^{-1} p(i) & x<0.       
      \end{cases}
\end{equation}
Then $P(x)$ is again a polynomial in $x$.
Let $\{-1,1\} \subset S$. Then, an $(S,l)$-tree has the form
 \begin{align*}
\begin{array}{ccccc}
    \ddots & & a & & \iddots \\
    & 0 & & l
  \end{array},
 \end{align*} 
where $0 \leq a \leq l$. Let $f(l,x)$ denote the number of $(S,l)$-trees such that the entry $a$ in the second row from the bottom is equal to $x$. By Theorem \ref{thm: s_t_trees}, $f(l,x)$ is a polynomial in $l$ and $x$. The weight $w_l(S)$ is given by
\[
 w_l(S)=\sum_{x=0}^l f(l,x).
\]
Hence for $l=-1$, \eqref{eq: sum of poly} implies
\[
 w_{-1}(S)=\sum_{x=0}^{-1} f(-1,x)=0.
\]
Assume $(l+1)|w_l(S)$. Since $S$ is irreducible and $w_l(S)$ has degree at least one, $\{-1,0,1\}$ must be a subset of $S$, which proves the first claim.\\

Let $\{-2,-1,1,2\} \subset S$. Then, an $(S,l)$-tree has the form
 \begin{align*}
\begin{array}{ccccccc}
\ddots & & a &  & b & &\iddots \\
   & -1 & & c & & l+1  \\
  &  & 0 & & l
  \end{array},
 \end{align*} 
with $-1 \leq a \leq c \leq b \leq l+1$, $0 \leq c \leq l$ and $a < b$.
Let $f(l,x,y)$ denote the number of $(S,l)$-trees where the entries $a,b$ in the third row from the bottom are equal to $x$ and $y$ respectively.
The weight function is given by
\[
w_l(S)=\sum_{c=0}^l\left( \sum_{a=-1}^c \sum_{b=c+1}^{l+1} f(l,a,b) +\sum_{a=-1}^{c-1} f(l,a,c) \right)
\]
For $l=-2$, the sum $\sum_{c=0}^{-2}$ is equal to $-\sum_{c=-1}^{-1}$. Thus, for the weight function we obtain
\begin{multline*}
w_{-2}(S)=-\left( \sum_{a=-1}^{-1} \sum_{b=0}^{-1} f(-2,a,b) +\sum_{a=-1}^{-2} f(-2,a,-1) \right)\\
= -\left( \sum_{a=-1}^{-1} 0 +0 \right)=0.
\end{multline*}
Assume $\{-2,-1,-1,2\}$ is not a subset of $S$. We can extend our definition of an $(S,l)$ tree using the generalisation of \eqref{eq: sum of poly}. The enumeration then becomes a weighted enumeration, where the weight of an $(S,l)$-tree is $(-1)$ to the power of encounters of entries $x_{i,j}> x_{i,j+1}$. Hence an $(S,-2)$-tree has the form
\begin{equation*}
\begin{array}{c c c c c c}
\ddots & & \vdots & & \iddots\\
& -1 && -1\\
&& 0 && -2
\end{array}  \quad ,\quad \qquad
\begin{array}{c c c c c c}
\ddots & & \vdots & & \iddots\\
& -1 && -1\\
 0 && -2
\end{array} \quad ,
\end{equation*}
where the left corresponds to an $(S,-2)$-tree with $\{-2,-1,0,1\} \subseteq S$ and the right to $\{-1,0,1,2\} \subseteq S$, and their weight is $-1$. By deleting the bottom row and adding 1 to all entries, we obtain an $(S^\prime,0)$-tree with $S^\prime=\{i+1|i \in S, i<-2\} \cup \{-1,0,1\} \cup \{i-1, i \in S, i>2\}$. Hence in both cases
\[
-w_{-2}(S) = |w_0(S^\prime)|
\]
holds. By the definition of an $(S,l)$-tree, it is obvious that for $l \geq 0$ there is at least one $(S,l)$-tree, hence
$w_0(S^\prime) \neq 0$, proving the claim.
\end{proof}

\subsection{Rational roots of $w_l(M)$}
\label{sec: roots of M}

\begin{prop}
\label{prop: adding a height one piece}
Let $M$ be a Motzkin path of length $n$ that ends with a south-east step, and let $M^\prime$ denote the Motzkin path of length $n+1$ obtained by putting an east step in front of the last step of $M$. The weight $w_l(M^\prime)$ is given by
\[
w_l(M^\prime)=(l+2n)w_l(M).
\]
\end{prop}

\begin{proof}
Let $A^\prime$ be an $(n+1,l)$-AS-trapezoid with $\M(A^\prime)=M^\prime$. If $A^\prime$ doesn't have a $-1$ in the second to last row from the top, the last two rows have the form
\[
\begin{array}{c	c c c c c}
0 &1 &0\, \cdots\, 0 &0 &0\, \cdots \,0 &0\\
&0 &0 \,\cdots\, 0 &1 &0 \,\cdots \,0
\end{array}
\]
up to horizontal and vertical reflection of the inner $l+2n-1$ columns. By reflecting the top two rows in such a way that one obtains the above form, and deleting the top row, we obtain an $(n,l)$-AS-trapezoid $A$ with $\M(A)=M$.
Now assume $A^\prime$ has a $-1$ entry in its second to last row from the top. Then the two top rows have the form 
\[
\begin{array}{c	c c c c c c c}
0 &0 &0\, \cdots\, 0 &1 &0\, \cdots \,0 &0 &0\, \cdots \,0 &0\\
&1 &0 \,\cdots\, 0 &-1 &0 \,\cdots \,0 &1 &0 \,\cdots \,0
\end{array}
\]
up to horizontal reflection. If we delete the top row and the leftmost $1$ and $-1$ in the second row from the top, we again obtain an $(n,l)$-AS-trapezoid $A$ with $\M(A)=M$.
Conversely, let $A$ be an $(n,l)$-AS-trapezoid. We can construct $4$ AS-trapezoids $A^\prime$ with no $-1$ in the second row from the top and $l+2(n-2)$ AS-trapezoids $A^\prime$ with one $-1$ entry in the second row from the top, such that $\M(A^\prime)=M^\prime$. This proves the claim.
\end{proof}

\begin{conj}
Let $M$ be a Motzkin path of length $n+k$. We let $M=(\ldots, e_1, \ldots, e_k)$ indicate that the last $k$ steps of $M$ are $e_1, \ldots, e_k$. Then the following hold.
\begin{align*}
&M= (\ldots, 1,0,-1,-1)  \Rightarrow (l+2n+5)|w_l(M),\\
&M= (\ldots, 1,0,0,-1,-1)  \Rightarrow (l+2n+7)|w_l(M),\\
&M= (\ldots, 1,1,-1,-1,-1)  \Rightarrow (l+2n+2)(l+2n+7)(l+2n+8)|w_l(M),\\
&M= (\ldots, 1,1,-1,0,-1,-1)  \Rightarrow (l+2n+2)(l+2n+7)|w_l(M),\\
&M= (\ldots, 1,1,-1,0,0,-1,-1)  \Rightarrow (l+2n+2)|w_l(M),\\
&M= (\ldots, 1,1,-1,0,0,0,-1,-1)  \Rightarrow (l+2n+2)|w_l(M),\\
&M= (\ldots, 1,0,1,-1,-1,-1)  \Rightarrow (l+2n+8)|w_l(M),\\
&M= (\ldots, 1,1,0,-1,-1,-1)  \Rightarrow (l+2n+2)(l+2n+8)^2|w_l(M),\\
&M= (\ldots, 1,1,0,-1,0,-1,-1)  \Rightarrow (l+2n+2)|w_l(M),\\
&M= (\ldots, 1,1,0,-1,0,0,-1,-1)  \Rightarrow (l+2n+2)|w_l(M),\\
&M= (\ldots, 1,1,0,-1,0,0,0,-1,-1)  \Rightarrow (l+2n+2)|w_l(M),\\
&M= (\ldots, 1,1,0,0,-1,-1,-1)  \Rightarrow (l+2n+2)(l+2n+10)|w_l(M).\\
\end{align*}
\end{conj}

It seems that the list of conjectures of the above kind can be continued as long as one likes. Experiments suggest that the steps $e_1, \ldots, e_k$ at the end must satisfy $e_1=1$ and $\sum_{i=1}^k e_i=-1$, i.e., the Motzkin path before these steps ``ends'' at height $1$. Further, it is interesting that all rational roots of the shortest Motzkin path with one of the above endings, with the exception of the path with the $7,9,12$-th ending, are explained by the above conjecture and Conjecture \ref{conj: roots of cCs}.

\section*{Acknowledgements}
The author would like to thank Arvind Ayyer for giving him access to his conjectures on the Motzkin path refinement of ASTs, which significantly influenced this work, as well as an anonymous referee for helpful suggestions and Yvonne Kemper for her assistance with the language.

\begin{appendix}

\section{tables for $w_l(S)$ and $w_l(M)$}
\label{app: Data}

The following two tables list the weight functions of all irreducible centred Catalan sets of size less than $6$ up to the reflection $S \mapsto \{-s:s\in S\}$, and the weight functions of all irreducible Motzkin paths up to length $5$.

\begin{tabular}{l  l }
            $S$ & $w_l(S)$\\
            \hline
              $\{0,1\}$ & $1$\vspace{3 pt}\\
	      $\{-1,0,1\}$ & $(l+1)$\vspace{3 pt}\\
	      $\{-1,0,1,2\}$ & $\frac{1}{2}(l+1)(l+4)$\vspace{3 pt}\\
	      $\{-2,-1,0,1,2\}$ & $\frac{1}{12}(l+1)(l+2)(l+6)(l+7)$ \vspace{3 pt}\\
	      $\{-3,-1,0,1,2\}$ & $\frac{1}{6}(l+1)(l+6)(2l+7)$\vspace{3 pt}\\
	      $\{-1,0,1,2,3\}$ & $\frac{1}{6}(l+1)(l+5)(l+6)$\vspace{3 pt}\\
	      $\{-2,-1,0,1,2,3\}$ & $\frac{1}{144}(l+1)(l+2)(l+7)(l^3+23l^2+168l+360)$\vspace{3 pt}\\
	      $\{-2,-1,0,1,2,4\}$ & $\frac{1}{24}(l+1)(l+2)(l+6)(l+7)(l+8)$\vspace{3 pt}\\
	      $\{-3,-1,0,1,2,3\}$ & $\frac{1}{24}(l+1)(l^4+25l^3+226l^2+864l+1176)$\vspace{3 pt}\\
	      $\{-1,0,1,2,3,4\}$ & $\frac{1}{24}(l+1)(l+6)(l+7)(l+8)$\vspace{3 pt}\\
	      $\{-2,-1,0,1,3,4\}$ & $\frac{1}{24}(l+1)(l+6)(3l^2+37l+92)$\vspace{3 pt}\\
	      $\{-3,-1,0,1,2,4\}$ & $\frac{1}{24}(l+1)(l+6)(5l^2+55l+132)$\vspace{3 pt}\\
	      $\{-4,-1,0,1,2,3\}$ & $\frac{1}{24}(l+1)(l+6)(l+8)(3l+13)$
    \end{tabular}
\newline
       
	\begin{tabular}{cl}
           $M$ & $w_l(M)$\\
           \hline
           \includegraphics[scale=0.5]{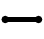} & $2$\vspace{2 pt}\\
           \includegraphics[scale=0.5]{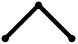} & $(l+1)$\vspace{2 pt}\\
           \includegraphics[scale=0.5]{m3} & $(l+1)(l+4)$\\
           \includegraphics[scale=0.5]{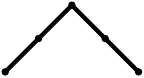} & $\frac{1}{12}(l+1)(l+2)(l+6)(l+7)$\vspace{2 pt}\\
           \includegraphics[scale=0.5]{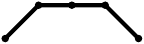} & $(l+1)(l+4)(l+6)$\\
           \includegraphics[scale=0.5]{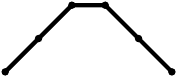} & $\frac{1}{72}(l+1)(l+2)(l+7)(l^3+23l^2+168l+360)$\\
           \includegraphics[scale=0.5]{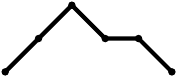} & $\frac{1}{12}(l+1)(l+2)(l+6)(l+7)(l+8)$\\
           \includegraphics[scale=0.5]{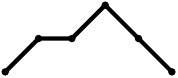} & $\frac{1}{12}(l+1)(l^4+25l^3+226l^2+864l+1176)$\vspace{2 pt}\\
           \includegraphics[scale=0.5]{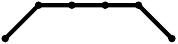} & $(l+1)(l+4)(l+6)(l+8)$\vspace{2 pt}
       \end{tabular}

\end{appendix}

\bibliographystyle{abbrv}
\bibliography{ASTs}

\end{document}